\documentclass[12pt]{article}
\usepackage[utf8]{inputenc}
\usepackage[mathscr]{euscript}
\usepackage{amsmath,amsfonts,amsthm,amssymb,tikz,paralist,multicol,comment,mathrsfs}
\usepackage[colorlinks,citecolor=red,urlcolor=blue,bookmarks=false,hypertexnames=true]{hyperref}
 \usepackage{tikz} 
 \usepackage{tikz-cd}

\newcommand{\R}{\mathbb{R}}

\newcommand{\Hom}{{\rm Hom}}
\newcommand{\bZ}{{\mathbb Z}}
\newcommand{\bC}{{\mathbb C}}
\newcommand{\bA}{{\mathbb A}}
\newcommand{\bR}{{\mathbb R}}
\newcommand{\bG}{{\mathbb G}}
\newcommand{\Spec}{{\rm Spec}\,}
\newcommand{\Pic}{{\rm Pic}\,}
\newcommand{\Val}{{\rm Val}\,}
\newcommand{\bfK}{{\mathbf K}}
\newcommand{\kerr}{\mathrm{ker}\,}
\newcommand{\bN}{{\mathbb N}}
\newcommand{\rank}{\mathrm{rank}\,}

\newtheorem{theorem}{\textbf{Theorem}}[section]
\newtheorem{proposition}[theorem]{\textbf{Proposition}}
\newtheorem{lemma}[theorem]{\textbf{Lemma}}
\newtheorem{corollary}[theorem]{\textbf{Corollary}}
\newtheorem{conjecture}[theorem]{\textbf{Conjecture}}

\newtheorem{definition}[theorem]{\textbf{Definition}}

\begin{document}

\title{Distribution of rational points on toric varieties: All the heights}
\author{Arda Demirhan and Ramin Takloo-Bighash}
\maketitle

Abstract: In this paper we prove a formula for the number of rational points in a certain height box for a split toric variety over a number field. 

\section{Introduction}
The purpose of this note is to provide an example where the role played by the exceptional set in Manin's conjecture can be eliminated. In its original formulation Manin's conjectures \cite{FMT, BM} dealt with a Fano variety over a number field and counted rational points with height relative to a metrized line bundle whose underlying line bundle was in the cone of effective divisors, and predicted an asymptotic formula for the number of rational points of bounded height on a small enough open set. One has to consider proper open sets as opposed to the whole variety because there are many examples where a proper Zariski closed subset of a variety can have more rational points than the rest of the variety. The easiest example one can construct is the blowup of $\mathbb P^2$ at the origin, \cite{Serre}, \S 2.12.  As stated these conjectures are wrong \cite{example}. This is due to the presence of dense thin sets with too many rational points. Recently, many works have explored the effects of thin sets on the distribution of rational points and their counting functions. In particular, Lehman, Sengupta, and Tanimoto \cite{LST, LT} have formulated a deep conjecture about the distribution of rational points on any geometrically rationally connected and geometrically integral smooth projective variety with bounded height relative to a metrized line bundle  which is big and nef (if $X$ is Fano, it suffices to assume that the line bundle is big). In this conjecture, there is a careful geometric description of a set $Z$ such that the number of rational points $x \in X(F) \setminus Z$ of height less than $B$ grows like 
$$
c B^a (\log B)^{b-1}
$$
with $c, a, b$ the expected constants, \cite{Peyre, Polarized}, provided that $X(F)$ is not thin. We refer the reader to \cite{LT, LST} for the precise formulation of the conjecture, examples, and references to many works by other mathematicians that motivated the conjecture. 

\

In the other direction, there have been attempts to formulate conjectures that do not require removing exceptional sets, \cite{Peyre2, Peyre3}.   In one approach by Peyre \cite[\S 4]{Peyre3}, the {\em all heights approach}, instead of considering a single height function, one considers the distribution of rational points relative to the height functions coming from {\em all } line bundles, or a family thereof. In the case where the cone of effective divisors is polyhedral, one could, for example, consider the distribution relative to the generators of the cone. Namely, suppose $X$ is a Fano variety over a number field $F$ and let $H_1, \dots, H_n$ be height functions associated to the generators of the cone of effective divisors of $X$. Fix positive real numbers $\beta_1, \dots, \beta_n$. For an open set $U$ and $B> 1$, we set 
$$
N(U, B) = \# \{ x \in U(F) \mid H_i(x) \leq B^{\beta_i}, 1 \leq i \leq n \}. 
$$
Then one might hope to prove an asymptotic formula $C B^a (\log B)^{b-1}$ for $N(U, B)$ with $a, b$ independent of $U$. There are, however, families of varieties of Picard rank $1$ with non-trivial exceptional sets. For some explicit constructions, see \cite[\S 10]{LT}. This means that the {\em all height} approach relative to the generators of the cone of effective divisors for this family of example is not going to remove the effect of exceptional sets on the asymptotic formulae in this case, and one may have to consider larger families of line bundles.  In general, it is not clear that the {\em all heights} approach for any family, and we would expect to need other notions, such as {\em freeness}, to formulate a Manin-type conjecture that does not require removing exceptional sets, \cite{Peyre2, Peyre3}. 

\

In this paper we consider the case of smooth projective toric varieties over a number field whose torus is split, the so-called {\em split toric varieties}, and we show that the all heights approach of Peyre works relative to the family of line bundles associated to all $T$-invariant divisors.  Manin's conjecture in its original form was proved for toric varieties by Batyrev and Tschinkel \cite{anisotropic, BT1, BT2} using harmonic analysis on tori. Toric varieties are the perfect test ground for investigations of the sort we consider here because their geometry is well-understood, and the structure of the boundary divisors is given in terms of combinatorial data.

\

We now state our theorem. Let $X$ be a split smooth projective toric variety defined over a number field $F$, and let $T$ be the open torus in $X$.  Let $\Delta_1, \dots, \Delta_n$ be the $T$-invariant divisors of $X$ as in \S\ref{Tdivisor}.

\begin{theorem}\label{theorem}  With notations as above,  there are metrizations of the line bundles corresponding to the divisors $\Delta_1, \dots, \Delta_n$ such that if $H_1, \dots, H_n$ are the associated height functions, then the following holds: Let $\beta_1, \dots, \beta_n$ be fixed positive real numbers. Then there are constants $C>0$ and $\epsilon > 0$ such that as $B \to \infty$, 
\begin{equation}\label{formula}
\# \{x \in X(F) \mid H_i(x) \leq B^{\beta_i}, 1 \leq i \leq n \}  = CB^{\sum_i  \beta_i} ( 1 + O(B^{-\epsilon})). 
\end{equation}
\end{theorem}

This theorem is Corollary \ref{coro:main} below, stated in the notations of \S\ref{Tdivisor}.  The proof has two parts. In the first part we count rational points on the torus and show that Equation \eqref{formula} holds. In the second part we show that the contribution of the boundary is less. This latter step relies on the fact that the boundary of a toric variety is a union of tori, and that we can use analogues of Equation \eqref{formula} for these tori in an inductive fashion.  In the case of toric varieties the exceptional set predicted by the conjecture of Lehman, Sengupta, and Tanimoto is contained in the boundary (Example 5.12 of \cite{LT}), and as indicated our proof shows that the boundary pieces contribute less than the open torus to the asymptotic formula. Unfortunately at present we do not know how to compute the constant $C$ in Theorem \ref{theorem}. 

\

To count rational points on the torus,  we use the methodology of Strauch and Tschinkel \cite{selecta} using height zeta functions.  The height zeta function considered by Strauch and Tschinkel is muti-dimensional, but they restrict the zeta function to a line and use a one dimensional Tauberian theorem. In our case, we do not restrict the zeta function to lines and apply a multi-dimensional Tauberian theorem, Theorem \ref{thm:dela}. The required analytic properties needed for the Tauberian theorem are given in Proposition \ref{prop:dela}. 

\

The splitness assumption in Theorem \ref{theorem} is made only to simplify notation, and to avoid dealing with Galois actions. In general, Equation \eqref{formula} has the same shape, except that the sum $\sum_i \beta_i$ has to be replaced by the similar sum over the Galois orbits of the boundary divisors.  The proof of general case follows from the same argument and requires only minor modifications.   

\

It might at first be surprising that there is no power of $\log$ in the statement of Theorem \ref{theorem},  given that most Manin-type results have a $\log$ term. We give two heuristic reasons for this phenomenon.  Consider, for example, the following situation: For a positive integer $n$, let $d(n)$ be the number of divisors of $n$. Then it is a well-known result that 
\begin{equation}\label{hyperbola1} 
\sum_{n \leq X} d(n) \sim X \log X 
\end{equation}
as $X \to \infty$. Now compare this with the following trivial statement
\begin{equation}\label{hyperbola2}
\# \{(x, y) \in \bZ^2 \mid 1 \leq x \leq X, 1 \leq y \leq Y \} = XY. 
\end{equation} 
Dirichlet's Hyperbola Method deduces the proof of Equation \eqref{hyperbola1} to Equation \eqref{hyperbola2}.  Theorem \ref{theorem} is the analogue of Equation \eqref{hyperbola2} in our context, and standard Manin type results should follow from it using general Dirichlet Hyperbola theorem type results. At present this is not possible as we do not have enough power saving in in the error term in Theorem \ref{theorem} to implement such an idea. 

\

Our next heuristic reasoning relies on the theory of {\em Universal Torsors}. For a split Fano toric variety over the field of rational numbers, there is a theory of universal torsors worked out by Salberger \cite{Salberger}. According to this theory in order to count rational points on a toric variety one needs to count integral points satisfying certain coprimality conditions on the universal torsor. The coordinates on the universal torsors are given, essentially, by the height functions $H_1, \dots, H_n$. For that reason, the expression on the left hand side of Equation \eqref{formula} counts integral points with certain coprimality conditions inside a box. Without the coprimality restrictions, the number of integral points in the box is approximately a constant multiple of $B$ raised to $\sum_i \beta_i$. Imposing the coprimality conditions, via an analysis involving Peyre's M\"obius function \cite[7.1.7]{Peyre} as in \cite[\S 11]{Salberger}, just modifies the constant in front of $B^{\sum_i \beta_i}$.  Because of the existence of units, the theory of universal torsors is far more complicated over fields other than the rational numbers. Pieropan \cite{Marta} has implemented Salberger's program for the case of imaginary quadratic fields taking advantage of the fact that the group of units is finite. In this case too a heuristic reasoning as above recovers Equation \eqref{formula}. Our approach makes no assumption about the number field.

\

We conclude this introduction by formulating two conjectures. Our first conjecture, involving general height boxes, is motivated by the above heuristic reasoning using universal torsors. 

\begin{conjecture}\label{conjecture-naive}  Let the notation be as in Theorem \ref{theorem}. There are metrizations of the line bundles corresponding to the divisors $\Delta_1, \dots, \Delta_n$ in such a way that if $H_1, \dots, H_n$ are the associated height functions, then the following holds: There are real numbers $C>0$ and $\epsilon>0$ such that for any fixed $a > 1$, and any real numbers $B_1, \dots, B_n > 0$ satisfying 
$$
\log B_i / \log B_j  \in (a^{-1}, a), \quad \text{ for all } i, j
$$ and $\min_i B_i \to + \infty$ we have 
\begin{equation}\label{formula-conjecturenaive}
\# \{x \in X(F) \mid H_i(x) \leq B_i, 1 \leq i \leq n \}  = C(\prod_i B_i) ( 1 + O(\prod_i B_i^{-\epsilon})). 
\end{equation}
\end{conjecture}

It would also be of interest to have a result similar to Theorem \ref{theorem} for the generators of the cone of effective divisors. In order to state the conjecture we need a couple of notations. Let $X$, $T$, and $F$ be as above.  Let $\Lambda_{\rm eff}(X)  \subset \Pic(X)_\R$ be the cone of effective divisors defined in \S \ref{cone}, and let $D_1, \dots, D_r$ be divisors whose classes form the extremal rays of the cone $\Lambda_{\rm eff}(X)$ in $\Pic(X)_\R$. Let $K_X \in \Pic(X)_\R$ be the canonical class. Write  $- K_X = \sum_{i=1}^r r_i D_i$, for real numbers $r_i > 0$.  Then we have the following conjecture:

\begin{conjecture}\label{conjecture}  There are metrizations of the line bundles corresponding to the divisors $D_1, \dots, D_r$ in such a way that if $H_{D_1}, \dots, H_{D_r}$ are the associated height functions, then the following holds: Let $\beta_1, \dots, \beta_r$ be fixed positive real numbers. Then there are constants $C'>0$ and $\epsilon > 0$ such that as $B \to \infty$, 
\begin{equation}\label{formula-conjecture}
\# \{x \in X(F) \mid H_{D_i}(x) \leq B^{\beta_i}, 1 \leq i \leq r \}  = C' B^{\sum_i  r_i \beta_i} ( 1 + O(B^{-\epsilon})). 
\end{equation}
\end{conjecture}

There should also be a statement similar to Conjecture \ref{conjecture-naive} for general height boxes associated to the generators of the cone of effective divisors. 

\

The proofs of these conjectures should follow an argument similar to that of Theorem \ref{theorem}, in that one should first prove Formula \eqref{formula-conjecture} with $X(F)$ replaced by $T(F)$, and then do induction. For Conjecture \ref{conjecture-naive} the induction step is identical to the one in the proof of Theorem \ref{theorem}, but the induction step for the proof of Conjecture \ref{conjecture} is rather subtle.  Unfortunately, however, the Tauberian theorem we use for the proof of Theorem \ref{theorem}, Theorem \ref{thm:dela}, is not strong enough to accomplish Equation \eqref{formula-conjecture} for the torus $T$.

\

This paper is organized as follow. Section \ref{sect:2} gives the basic definitions regarding toric varieties and their associated height functions. Section \ref{sect:3} contains the analytic properties of the height zeta function. Section \ref{sect:4} has the proof of the main theorem.

\

This work owes a great deal of intellectual debt to Emmanuel Peyre and Yuri Tschinkel. We thank them both for useful conversations and for their willingness to share their insights. We also wish to thank Tim Browning, Jordan Ellenberg, Sho Tanimoto, and Akshay Venkatesh, and the anonymous referee for useful communications and comments.  The second author is partially supported by a Collaboration Grant from the Simons Foundation. 

\section{Preliminiaries}\label{sect:2}
The standard elementary introduction to the geometry of toric varieties is \cite{Fulton}. We use the notations and terminology of \cite{selecta}. In order to make this paper as self-contained as possible we review the basic theory of split toric varieties and their height functions in this section. 

\subsection{Toric varieties} 

\subsubsection{} Let $T$ be a split algebraic torus of dimension $d$ over a number field $F$, and let $M = X^*(T), N = \Hom(M, \bZ)$.  We 
let $\Sigma$ be a complete regular 
fan in $N_\bR$, and we define $X= X_\Sigma$ to be the smooth projective toric variety corresponding to $\Sigma$. The variety $X$ is 
obtained by patching together affine open sets 
$$
U_\sigma = \Spec(F[M \cap \check \sigma])
$$
for $\sigma \in \Sigma$, and for each $\sigma$, 
$$
\check \sigma = \{ m \in M_\bR \mid n(m) \geq 0, \forall n \in \sigma \}. 
$$

\subsubsection{} We define $PL(\Sigma)$ be the group of functions $\varphi$ on $N_\bR$ that satisfy the following two conditions: (i) $\varphi(N)
\subset \bZ$, (ii) for each $\sigma \in \Sigma$, the restriction of $\varphi$ is equal to the restriction of a linear function on $N$ to 
$\sigma$. We call $PL(\Sigma)$ the group of  $\Sigma$-piecewise linear integral functions on $N_\R$. For $\varphi \in 
PL(\Sigma)$ and every $d$ dimensional $\sigma \in \Sigma$, there is a unique $m_{\varphi, \sigma} \in M$ such that 
$$
\varphi(n) = n(m_{\varphi, \sigma}). 
$$
For a general $\sigma \in \Sigma$, we pick a $d$ dimensional $\sigma' \in \Sigma$ and we set
$$
m_{\varphi, \sigma} = m_{\varphi, \sigma'}. 
$$
To any $\varphi \in PL(\Sigma)$ we associate a line bundle $L_\varphi$ on $X$ by setting 
$$
L_\varphi = \left(\coprod_{\sigma \in \Sigma} U_\sigma \times \bA^1 \right)/\sigma
$$
where $(x, a) \in U_\sigma \times \bA^1$ and $(x', a') \in U_{\sigma'} \times \bA^1$ are identified if $x =x' \in U_\sigma \cap
U_{\sigma'}$ and $a' = \frac{m_{\varphi, \sigma'}}{m_{\varphi,\sigma}}(x) \cdot a$.  The class of $(x, a)$ will be denoted by 
$[x, a]$.  There is a canonical $T$ action on $L_\varphi$ given by
$$
[x, a] \cdot t = [x \cdot t, m_{\varphi, \sigma}(t) \cdot a],
$$
if $x \in U_\sigma$. This action identifies $PL(\Sigma)$ with $T$-linearized line bundles on $X$. We have an exact 
sequence 

\begin{equation} \label{sequence}
0 \to M \to PL(\Sigma) \to \Pic X \to 0
\end{equation}
where the map $\eta: PL(\Sigma) \to \Pic X$ is given by $\varphi \mapsto L_\varphi$.  We let $PL(\Sigma)_\R = PL(\Sigma) \otimes \R$, $\Pic(X)_\R = \Pic (X) \otimes \R$. We extend $\eta$ linearly to a map $\eta_\R: PL(\Sigma)_\R \to \Pic(X)_\R$.  

\subsubsection{}\label{cone}
We define {\em the cone of effective divisors} $\Lambda_{\rm eff}(X)$ to be the cone in $\Pic(X)_\R$ generated by classes of effective divisors on $X$. Let us describe the cone $\Lambda_{\rm eff}(X)$.  We define  $\Sigma_1 \subset N$ to be the set of primitive integral generators of the one-dimensional cones in $\Sigma$ 
and we set 
$$
PL(\Sigma)^+ = \{ \varphi \in PL(\Sigma)_\bR \mid \varphi(e) > 0, \forall e \in \Sigma_1 \}.
$$
It is known that 
\begin{equation}\label{eq:effective}
\Lambda_{\rm eff}(X) = \overline{\eta_\R(PL(\Sigma)^+)}
\end{equation}
where the closure is taken in $\Pic(X)_\bR$.  A fact of utmost importance is the statement that if we define $\varphi_\Sigma$
to be the piecewise-linear function which is defined by $\varphi_\Sigma (e) =1$ for all $e \in \Sigma_1$, then 
$L_{\varphi_\Sigma}$ is isomorphic to the anti-canonical line bundle on $X$. 

\subsubsection{}\label{Tdivisor}
For each $e \in \Sigma_1$, let $T_e$ be the $(d-1)$-dimensional torus orbit corresponding to the cone $\bR_{\geq 0}e \in \Sigma$. Define $D_e$ to be the Zariski closure of $T_e$ in $X$.  Proposition 1.2.11 of \cite{anisotropic}
asserts that the images of $D_e$, $e \in \Sigma_1$, generate the cone $\Lambda_{\rm eff}(X)$. Each divisor $D_e$ is a toric variety in its own right.  Let $(N_e)_\bR = N_\bR/ \bR e$, and let $\Sigma_e$ be the cones in $\Sigma$ that contain $\bR_{\geq 0} e$ as a face.  Then for each $\sigma \in \Sigma_e$, we let $\bar \sigma = \sigma + \bR e / \bR e$. The collection $\{ \bar \sigma, \sigma \in \Sigma_e\}$, again denoted $\Sigma_e$, is a fan in $(N_e)_{\bR}$. The toric variety corresponding to this fan is $D_e$. 

\subsection{Metrizations} Following \cite{selecta} we introduce an adelic metric on $L_\varphi$. For $\sigma \in \Sigma$
and $v \in \Val(F)$ we set 
$
\bfK_{\sigma, v} = \{ x \in U_\sigma(F_v) \mid |m(x)|_v \leq 1, \forall m \in \check\sigma \cap M\}.
$
If $x \in \bfK_{\sigma, v}, a \in F_v$, then $[x, a] \in L_\varphi(F)$ is $F_v$-rational, and we put 
$$
\| [x, a] \|_x =|a|_v. 
$$
This is well-defined. The family $\|\cdot \|_v = (\| \cdot \|_x)_{x \in X(F_v)}$ is a $v$-adic metric on $L_\varphi$, and 
$\mathcal L_\varphi = (L_\varphi, (\|\cdot \|_v)_v)$ is a metrization of $L_\varphi$. Let $\bfK_{T, v}$  be the maximal 
compact subgroup. For $x \in T(F_v)$, consider the map $M \to \bZ$, for $v$ non-archimedean, and $M \to \bR$ for $v$ 
archimedean, given by 
$$
m \mapsto - \frac{\log ( |m(x)|_v)}{\log q_v}
$$
where for $v$ non-archimedean, $q_v$ is the size of the residue field of $F_v$, and for $v$ archimedean, we set 
$\log q_v =1$. This map provides an isomorphism $T(F_v)/\bfK_{T, v} \to N$ for $v$ non-archimedean, and $N_\bR$ 
otherwise. Let $\bar{x}$ be the image of $x \in T(F_v)$ in $N$, resp. $N_\bR$ for archimedean $v$. For $\varphi
\in PL(\Sigma)_\bC$ define a function $H_{\Sigma, v}(\cdot, \varphi): T(F_v) \to \bC$ by 
$$
H_{\Sigma, v}(x, \varphi) = e^{-\varphi(\bar x) \log(q_v)}.
$$
The corresponding global height function $H_\Sigma(\cdot, \varphi): T(\bA) \to \bC$ defined by 
$$
H_\Sigma(x, \varphi) = \prod_v H_{\Sigma, v}(x_v, \varphi)
$$
is well-defined. 
\subsection{$\mathcal A_T$ and $\mathcal U_T$}
Let $\bfK_T = \prod_v \bfK_{T, v} \subset T(\bA)$. Define
$$
\mathcal A_T = (T(\bA)/T(F)\bfK_T)^*
$$
to be the group of unitary characters of $T(\bA)$ which are trivial on $T(F)\bfK_T$. For $m \in M$ we 
define the character $$\chi^m(x) = e^{i \log(|m(x)|_\bA)}.$$
The map $m \mapsto \chi^m$ provides us with an embedding $M_\bR \to \mathcal A_T$. For any archimedean place
$v$ and $\chi \in \mathcal A_T$ there is $m_v = m_v(\chi) \in M_\bR$ such that 
$$\chi_v(x_v) = e^{-i \bar x_v(m_v)}$$
for $x_v \in T(F_v)$. We then get a homomorphism 
$$
\mathcal A_T \to M_{\bR, \infty} = \oplus_{v \mid \infty} M_\bR, 
$$
given by sending $\chi$ to the vector $m_\infty(\chi) = (m_v(\chi))_{v \mid \infty}$.  
In the sequel, given a real vector space $W$, we set $W_\infty = \oplus_{v \mid \infty} W$.

We define $T(\bA)^1$ to be the common
kernel of the all maps $T(\bA) \to \bR_{>0}$, $x\mapsto |m(x)|_\bA$, $m \in M$. Set 
$$
\mathcal U_T = (T(\bA)^1/T(F)\bfK_T)^*. 
$$
If we fix an isomorphism $T \tilde\longrightarrow \, \mathbb G_{m, F}^d$ and a projection $\mathbb G_m(\bA) \to \mathbb
G_{m}(\bA)^1$ we obtain a splitting of the exact sequence 
$$
1 \rightarrow T(\bA)^1 \rightarrow T(\bA) \rightarrow T(\bA)/T(\bA)^1 \to 1. 
$$
We then obtain a decomposition 
$$
\mathcal A_T = M_\bR \oplus \mathcal U_T
$$
and
\begin{equation}\label{sum}
M_{\bR, \infty} = M_\bR \oplus M_{\bR, \infty}^1
\end{equation}
with $M_{\bR, \infty}^1$ being the minimal $\R$-subspace of $M_{\bR, \infty}$ containing the image of $\mathcal U_T$
under the map $\mathcal A_T \to M_{\bR, \infty}$ defined above. For the remainder of this work we fix these splittings. Note
that by Dirichlet's unit theorem, the image of $\mathcal U_T \to M_{\bR,\infty}^1$ is a lattice of maximal rank. 

\

We also fix measures as follows. For an archimedean place $v$, let $dx_v$ be the Haar measure on $T(F_v)$ 
giving $\bfK_{T, v}$ volume one. For archimedean $v$, we give $T(F_v)/\bfK_{T, v}$ the pullback of the Lebesque 
measure on $N_\bR$, giving volume $1$ to $N_{\bR}/N$. We then give $\bfK_{T, v}$ the Haar measure with total 
measure one.  We then obtain an invariant measure on $T(F_v)$.  On $T(\bA)$ we have an adelic measure 
$dx = \prod_v dx_v$. 

\section{Height zeta function}\label{sect:3}
In this section we establish the basic properties of the height zeta function. 
\subsection{Integrals of height functions}\label{subsect:integral}
For a character $\chi: T(F_v) \to S^1$ define 
$$
\hat H_{\Sigma, v}(\chi, \varphi) = \int_{T(F_v)} H_{\Sigma, v}(x_v, \varphi)\chi(x_v) dx_v. 
$$
Suppose the integral is convergent. Then, as $H$ is invariant under $\bfK_{T, v}$, the value of the integral is zero unless $\chi$ is trivial on $\bfK_{T, v}$. These integrals converge absolutely for $\varphi \in PL(\Sigma)^+$.

\

Let $v$ be an archimedean place of $F$. Any $d$ dimensional cone $\sigma \in \Sigma$ is simplicial, generated by 
$\Sigma_1 \cap \sigma$. Let $\chi$ be unramified. This means $\chi(x) = e^{- \bar{x}(m)}$ for some $m \in M_\bR$. 
Then 
$$
\hat H_{\Sigma, v} (\chi, \varphi) = \sum_{\dim \sigma = d} \prod_{e \in \sigma \cap \Sigma_1} \frac{1}{\varphi(e) + i e(m)}. 
$$
We also compute the integral at the non-archimedean places. We have a set of variables $\underline{u}=
\{u_e\}_{e \in \Sigma_1}$. Set 
for each $\sigma \in \Sigma$
$$
R_\sigma((u_e)_e) = \prod_{e \in \sigma \cap \Sigma_1} \frac{u_e}{1-u_e}, 
$$
and define 
$$
R_\Sigma((u_e)_e) = \sum_{\sigma \in\Sigma} R_\sigma((u_e)_e), 
$$
and
$$
Q_\Sigma((u_e)_e) = R_\Sigma((u_e)_e) \prod_{e \in \Sigma_1} (1-u_e). 
$$
Then if $\chi$ is an unramified unitary character of $T(F_v)$ and if $\Re \phi \in PL(\Sigma)^+$, then 
$$
\hat H_{\Sigma, v}(\chi, \varphi) = Q_\Sigma((\chi(e) q_v^{-\varphi(e)}) \prod_{e \in \Sigma_1} 
(1- \chi(e) q_v^{-\varphi(e)})^{-1}.
$$
Any $e \in\Sigma_1$ defines a homomorphism $F[M] \to F[\bZ]$, and by duality, a morphism of algebraic tori 
$\mathbb G_m \to T$. For any character $\chi \in\mathcal A_T$, we let $\chi_e$ be the Hecke character 
$$
\bG (\bA) \to T(\bA) \rightarrow S^1. 
$$
Let 
$$
L_f(\chi_e, s) = \prod_{v \nmid \infty} (1 - \chi_e(\pi_v)q_v^{-s})^{-1}. 
$$
The product converges for $\Re s >1$. A consequence of this is that the global Fourier transform 
$$
\hat H_\Sigma(\chi. \varphi) = \int_{T(\bA)} H_\Sigma(x, \varphi)\chi(x) \, dx 
$$
converges absolutely if $\Re \varphi \in\varphi_\Sigma + PL(\Sigma)^+$.

\subsection{An inequality}
 For $\chi \in \mathcal A_T$
and $\Re \varphi \in \frac{1}{2} \varphi_\Sigma + PL(\Sigma)^+$, we put 
$$
\zeta_\Sigma(\chi, \varphi) = \prod_{v \mid \infty} \hat H_{\Sigma, v}(\chi_v, \varphi) \prod_{v \nmid \infty} Q_\Sigma 
((\chi_v(e)q_v^{-\varphi(e)})_e).
$$
The following lemma follows from the proof of Prop. 2.3.2 of \cite{anisotropic}. 
\begin{lemma}
Let $\bfK$ be a compact subset of $\frac{1}{2} \varphi_\Sigma + PL(\Sigma)^+$, and let $\mathcal T_\bfK \subset 
PL(\Sigma)_\bC$ be the tube domain over $\bfK$. Then there is a constant $c = c(\bfK)$ such that for all 
$\varphi \in \mathcal T_\bfK$, 
$$
|\zeta_\Sigma(\chi, \varphi)| \leq c \prod_{v \mid \infty} \left\{
\sum_{\dim \sigma = d} \prod_{e \in \sigma \cap \Sigma_1} \frac{1}{(1 + |e(\Im \varphi) + e(m_v(\chi)|)^{1 + 1/d}}
\right\}.
$$
\end{lemma}
For $\Re \varphi \in \varphi_\Sigma + PL(\Sigma)^+$, we have 
$$
\hat H_\Sigma(\chi, \varphi) = \zeta_\Sigma (\chi, \varphi) \prod_{e \in \Sigma_1} L_f(\chi_e, \varphi(e)).
$$
The following lemma follows from the proof of Lemma 5.4 of \cite{selecta}.

\begin{lemma}
There is a convex open neighborhood $B$ of the origin in $PL(\Sigma)_\bR$ such that the following holds: For any compact $\bfK \subset B$ there is a constant $\kappa(\bfK) > 0$ such that for any $\varphi \in \mathcal T_\bfK$, $\chi \in \mathcal U_T$ we have
$$
\left| \hat H_\Sigma(\chi, \varphi + \varphi_\Sigma) \prod_{e \in \Sigma_1} \frac{\varphi(e)}{\varphi(e)+1} \right| 
$$
$$
\leq\kappa(\bfK) \prod_{v \mid \infty} \left\{
\sum_{\dim \sigma = d} \prod_{e \in \sigma \cap \Sigma_1} \frac{1}{(1 + |e(\Im \varphi) + e(m_v(\chi)|)^{1 + 1/2d}}
\right\}.
$$
\end{lemma}

\subsection{The zeta function}\label{subsect:zeta}

 It is Proposition 3.4 of \cite{selecta} that 
the series 
$$
\sum_{x \in T(F)} H_\Sigma(x, \varphi) 
$$
converges absolutely and uniformly for $\Re \varphi \in \varphi_\Sigma + PL(\Sigma)^+$. We define the height zeta function $Z_\Sigma(\varphi)$, first for $\varphi \in \varphi_\Sigma +  PL(\Sigma)^+$, by setting
$$
Z_\Sigma(\varphi) = \sum_{x \in T(F)} H_\Sigma(x, \varphi). 
$$
In order to apply the Tauberian theorem of \S\ref{subsect:tauberian} we need to analytically continue this function. 

\

By the Poisson summation formula,
\begin{equation}\label{eq:poisson}
\sum_{x \in T(F)} H_\Sigma(x, \varphi) = \mu_T  \int_{M_\bR} \left\{
\sum_{\chi \in \mathcal U_T} \hat H_\Sigma(\chi, \varphi + im) 
\right\} \, dm, 
\end{equation}
where, if the Lebesgue measure is normalized to give $M_\bR/M$ volume one, then 
$$
\mu_T = \frac{1}{(2\pi \kappa)^d}, \quad \kappa = \frac{h_F\cdot R_F}{w_F}
$$
with $h_F$ the class number, $R_F$ the regulator, and $w_F$ the number of roots of unity of $F$. 

%

As in the proof of Theorem 5.5 of \cite{selecta} we define a function $c_0: PL(\Sigma)_{\bR, \infty} \to \bR_{>0}$ by 
$$
c_0((\varphi_v)_v) = \prod_{v \mid \infty} \left\{ \sum_{\dim \sigma = d} \prod_{e \in \sigma \cap \Sigma_1} 
\frac{1}{(1+|\varphi_v(e)|)^{1+1/2d}}\right\}.
$$
Let $\mathcal F \subset M_{\bR, \infty}^1$ be the fundamental parallelogram constructed from a basis of the image 
of $\mathcal U_T$ in $M_{\bR, \infty}^1$. There is a constant $c'>0$ such that for all $m_\infty(\chi \in M_{\bR, \infty}^1$, $\chi \in \mathcal U_T$, and all $m^1 \in \mathcal F$, 
$$
c_0(m_\infty(\chi)) \leq c' c_0(m_\infty(\chi) + m^1).
$$
Let $dm^1$ be the Lebesgue measure on $M_{\bR, \infty}^1$ normalized by the image of $\mathcal U_T$, and set
$$
c((\varphi_v)_v) = h_F^d c' \int_{M_{\bR, \infty}^1} c_0((\varphi_v)_v + m^1) \, dm^1. 
$$
Then we have 
$$
\sum_{\chi \in \mathcal U_T} c_0((\varphi_v)_v + m_\infty(\chi)) \leq c( (\varphi_v)_v). 
$$
Since the latter function is defined $ \mod M_{\bR, \infty}^1$, let us describe $PL(\Sigma)_{\bR, \infty} / M_{\bR, \infty}^1$. By \eqref{sequence}, $PL(\Sigma)_\R = M_\R \oplus \Pic(X)_\R$.  This means $PL(\Sigma)_{\bR, \infty} = M_{\bR, \infty} \oplus \Pic(X)_{\bR, \infty}$. By \eqref{sum}, $M_{\bR, \infty} = M_\bR \oplus M_{\bR, \infty}^1$. This means
$$
PL(\Sigma)_{\bR, \infty} = \Pic(X)_{\bR, \infty} \oplus M_\bR \oplus M_{\bR, \infty}^1. 
$$

If $\varphi \in PL(\Sigma)_\bR$, we denote by $(\varphi)_v$  the vector in $PL(\Sigma)$ all of whose coordinates are equal to $\varphi$.  We abbreviate $c((\varphi)_v)$ to $c(\varphi)$. 

\begin{lemma}\label{lemma-inequality}
There is a convex open neighborhood $B$ of the origin in $PL(\Sigma)_\bR$ such that the following holds: For any compact $\bfK \subset B$ there is a constant $\kappa(\bfK) > 0$ such that for any $\varphi \in \mathcal T_\bfK$ we have
$$
\left|   \left\{
\sum_{\chi \in \mathcal U_T} \hat H_\Sigma(\chi, \varphi + \varphi_\Sigma) 
\right\}\prod_{e \in \Sigma_1} \frac{\varphi(e)}{\varphi(e)+1} \right| \leq
\kappa(\bfK) c( \Im \varphi).
$$
\end{lemma}

\subsection{Meromorphic continuation of certain integrals}\label{subsect:distinguished}
Suppose $E$ is a finite dimensional vector space $\bR$ and $E_\bC$ its complixification. Also suppose $V \subset E$ is 
a subspace of dimension $d$, and $l_1, \dots, l_m$ are linearly independent real valued forms on $E$.  For each $j$, let $H_j$ be the kernel 
of $l_j$. Let $B \subset E$ be an open and convex neighborhood of $0$ such that for each $j$, and all $x \in B$, 
$l_j(x) > -1$. Let $\mathcal T_B$ be the tube domain over $B$. 
We start with the following definition: 
\begin{definition}
A function $c: V \to \bR_{\geq 0}$ is called sufficient if it satisfies
\begin{enumerate}
\item For any subspace $U \subset V$ and any $v \in V$ the functions $U\to \bR$ given by $u \mapsto c(v+u)$ is 
measurable on $U$ and the integral
$$
c_U(v) = \int_U c(v+u) \, du 
$$
is finite. 
\item For any subspace $U \subset V$ and $v \in V \setminus U$, 
$$
\lim_{\tau \to \pm \infty} c_U(\tau.v) =0.
$$
\end{enumerate}
\end{definition}

\begin{definition} We call a meromorphic function $f: \mathcal T_B \to \bC$
distinguished with respect to $(V; c; l_1, \dots, l_m)$ if 
\begin{enumerate}
\item The function
$$
g(z) = f(z) \prod_{j=1}^m \frac{l_j(z)}{l_j(z) +1}
$$
is holomorphic on $\mathcal T_B$; 
\item There is a sufficient function $c:V \to \bR_{\geq 0}$ such that for all compact sets $\bfK \subset \mathcal T_B$, 
all $z \in \bfK$ and $v \in V$, 
$$
|g(z + iv)| \leq \kappa(\bfK) c(v). 
$$
\end{enumerate}
\end{definition}
 Let $C$ be a connected component of $B \setminus \cup_{j=1}^m H_j$ and $\mathcal T_C$ the tube domain over $C$. 
 Set 
 $$
 \tilde f_C(z) = \frac{1}{(2\pi)^\nu} \int_V f(z + iv) \, dv, $$
 $\nu = \dim V$. 
 
\begin{proposition}\label{prop:distinguished}
Suppose $f$ a distinguished function with respect to $(V; c; l_1, \dots, l_m)$. Then 
\begin{enumerate}
\item $\tilde f_C(z): \mathcal T_C \to \bC$ is a holomorphic function; 
\item There is an open and convex neighborhood $\tilde B$ of $0$ containing $C$, and linear forms $\tilde l_1, \dots, 
\tilde l_{\tilde m}$, all vanishing on $V$, such that 
$$
z \mapsto \tilde f_C(z) \prod_{j=1}^{\tilde m} \tilde l_j(z)
$$ 
has a homomorphic continuation to $\mathcal T_{\tilde B}$.  Moreover, for each $j$, $\kerr(\tilde l_j) \cap C = 
\varnothing$. 
\end{enumerate}
\end{proposition}

\subsection{A limit computation}
Let notation be as in \S \ref{subsect:distinguished}.   Let $E^{(0)} = \cap_{j=1}^m \kerr (l_j)$ and $E_0 = E/E^{(0)}$. Suppose $V \cap E^{(0)} = \{ 0 \}$.   Let 
$$
E_0^+ = \{ x \in E_0 \mid l_j(x) \geq 0, 1 \leq j \leq m \}. 
$$ 
Let $\pi_0:E \to E_0$  and $\psi_0 : E_0 \to E_0^+ / \pi_0(V)$ be the canonical projections. Let $P = E_0^+ / \pi_0(V)$.  Assume  $\pi_0(V) \cap E_0^+ = \{0\}$. With these assumptions $\Lambda = \psi_0(E_0^+)$ is a strictly convex polyhedral
cone. Let $dy$ be the Lebesgue measure on $E_0^\vee$ normalized by the lattice $\oplus_{j=1}^m \bZ l_j$. 
Let $A \subset V$ be a lattice and $dv$ the Lebesgue measure on $V$ noramalized by the lattice $A$. On 
$V^\vee$ we have a Lebesgue measure $dy'$ normalized by $A^\vee$ and a measure $dy''$ on $P^\vee$ constructed using the a section of the projection $E_0^\vee \to V^\vee$ such that $dy = dy' \, dy''$. 

\

Define the $\mathcal X$-function of the cone $\Lambda$ by
$$
\mathcal X_\Lambda(x) = \int_{\Lambda^\vee}  e^{y''(x)} \, d y''
$$ 
for $x \in P_\bC$ with $\Re x$ in the interior of the cone $\Lambda$.  Let 
$$
B^+ = B \cap \{ x \in E \mid \l_j(x) >0, 1 \leq j \leq m \}. 
$$
We note that $\tilde f_C$ is holomorphic on $\mathcal T_{B^+}$. The following is Theorem 5.3 of \cite{selecta}:
\begin{proposition}\label{prop:limit} For $x_0 \in B^+$,  we have 
$$
\lim_{s \to 0} s^{m - \nu} \tilde f_C(s x_0) = g(0) \mathcal X_\Lambda(\psi_0(x_0)). 
$$
\end{proposition}

\subsection{A sufficient function}
Let $V$ be a $d$-dimensional  vector space over $\bR$ and let $(l_1, \dots, l_d)$ be a basis of $V^\vee = \Hom(V, \bR)$.  Fix a constant $\epsilon > 0$, and define a function $c: V \to \bR_{>0}$ by setting 
$$
c(v) = \prod_{j=1}^d \frac{1}{(1+ |l_j(v)|)^{1+\epsilon}}.
$$
We recall Theorem 7.6 of \cite{selecta}.
\begin{proposition}\label{thm:sufficient}
Let $U \subset V$ be a subspace and $du$ a Lebesgue measure on $U$. Then there is a positive constant $\kappa = \kappa(\epsilon, U, du)$ and a finite family $\{(l_{\alpha, 1}, \dots, l_{\alpha, d'})\}_{\alpha \in A}$ of bases of 
$(V/U)^\vee$, $d' = \dim V/U$, such that for all $v \in V$ we have 
$$
\int_U c(v + u) \, du \leq \kappa \sum_{\alpha \in A} \prod_{j=1}^{d'} \frac{1}{(1 + |l_{\alpha, j}(v)|^{1+ \epsilon}}.
$$
\end{proposition}
\subsection{Meromophic continuation of the height zeta function} 

We now return to the notations of \S\ref{subsect:zeta}. By Proposition \ref{thm:sufficient}, when $E = PL(\Sigma)_\bR$ and $V = M_\bR$,  the function $c$ is sufficient.  Equation \eqref{eq:poisson} allows us to invoke the machinery of 
\S \ref{subsect:distinguished}. By Proposition \ref{prop:distinguished}, the height zeta function $Z_\Sigma(\varphi)$, a priori only holomorphic on $\varphi_\Sigma + PL(\Sigma)^+$, has a meromorphic continuation to a domain 
$\varphi_\Sigma + \mathcal T_{\tilde B}$, with $\tilde B$ an open and convex neighborhood of the origin. 
Furthermore, there are linear forms $\tilde l_1, \dots, \tilde l_{\tilde m}$, all vanishing on $M_\R$, such that 
$$
F(\varphi) = Z_\Sigma (\varphi + \varphi_\Sigma) \prod_{j=1}^{\tilde m} \tilde l_j(\varphi)
$$
has a holomorphic continuation to $\mathcal T_{\tilde B}$.  By the proof of Proposition \ref{prop:distinguished}, the linear forms $\tilde l_j$'s are linearly independent.  Since the $\tilde l_j$'s are real forms on $PL(\Sigma)_\bR$ which vanish on $M_\bR$, by Equation \eqref{sequence}, they descend to $\Pic(X)_\bR$.    We also note that since $Z_\Sigma (\varphi + \varphi_\Sigma)$ is holomorphic on $PL(\Sigma)^+$, the descent of $\tilde l_j$'s to $\Pic(X)_\bR$ is positive inside the cone of effective divisors given by Equation \eqref{eq:effective}, and non-negative on the cone.

Next, by Proposition \ref{prop:limit}, if $ m = \dim PL(\Sigma)_\bR$ and $\nu = \dim M_\bR$, for $\varphi_0 \in PL(\Sigma)^+$, the limit $\lim_{s \to 0} s^{m-\nu} F(s\varphi_0)$ is non-zero if 
$$
\lim_{\varphi \to 0}  \left\{
\sum_{\chi \in \mathcal U_T} \hat H_\Sigma(\chi, \varphi + \varphi_\Sigma) 
\right\}\prod_{e \in \Sigma_1} \frac{\varphi(e)}{\varphi(e)+1} \ne 0. 
$$
We let $s$ approach $0$ through $PL(\Sigma)^+$ to get absolutely convergent integrals. By the computations of \S \ref{subsect:integral} the only $\chi$'s which contribute to this limit are those such that $\chi_e = 1$ for all $e \in \Sigma_1$.  It is easy to see that such $\chi$ will have to be trivial. For trivial $\chi$ the limit 
$$
\lim_{\varphi \to 0}  \hat H_\Sigma (1, \varphi + \varphi_\Sigma)  \cdot
\prod_{e \in \Sigma_1} \frac{\varphi(e)}{\varphi(e)+1}
$$
exists and is non-zero.  This computation means that $\tilde m = \dim \Pic(X)_\bR$. 

\

We can take the open set $B$ in such a way that for each $j$ and $x \in B$, $\tilde l_j(x) > -1$. It is then a consequence of the proof of Proposition \ref{prop:distinguished} and Proposition \ref{thm:sufficient} that 
the analytic continuation of 
$$
Z_\Sigma(\varphi + \varphi_\Sigma) \cdot \prod_{j=1}^{\tilde m} \frac{\tilde l_j(\varphi)}{\tilde l_j(\varphi)+1}
$$
is bounded on the tube domain $\mathcal T_B$. 

\

We summarize these results as the following proposition. 

\begin{proposition}\label{prop:dela}
The shifted  height zeta function $Z_\Sigma(\varphi + \varphi_\Sigma)$ is holomorphic on $PL(\Sigma)^+$.  There are linearly independent real linear forms $\tilde l_j$, $1 \leq j \leq \dim \Pic (X)_\bR$, vanishing on $M_\bR$, such that 
$$
Z_\Sigma(\varphi + \varphi_\Sigma) \cdot \prod_{j=1}^{\tilde m} \frac{\tilde l_j(\varphi)}{\tilde l_j(\varphi)+1}
$$
has an analytic continuation to a bounded holomorphic function on a tube domain $\mathcal T_B$ with $B$ an open convex neighborhood of $0$. The linear forms $\tilde l_j$ are positive inside the cone of effective divisors of $X$, and non-negative on the cone.  
\end{proposition}

\section{Tauberian argument}\label{sect:4}

In this section we review a Tauberian theorem of de la Bret\`{e}ch \cite{dela1, dela2} and use it to prove our main theorem. 

\subsection{A theorem of de la Bret\`{e}che}\label{subsect:tauberian}
In this section we recall the following theorem of de la Bret\`{e}che. 
\begin{theorem}\label{thm:dela}
Let $\mathcal S_1, \dots, \mathcal S_m$ be discrete strictly increasing sequences of real numbers larger than or equal to $1$.  
Suppose $f$ is a positive function defined on $\mathcal S_1 \times \cdots \times \mathcal S_r$, and let $F$ be the associated multiple 
Dirichlet series
$$
F(\mathbf s) = \sum_{d_1 \in \mathcal S_1} \cdots \sum_{d_r \in \mathcal S_r} \frac{f(d_1, \dots, d_r)}{d_1^{s_1} \cdots d_r^{s_r}}.
$$
Suppose there is a vector $\mathbf \alpha \in (0, + \infty)^r$ such that $F$ satisfies the following conditions: 
\begin{enumerate}
\item[(P1)] The series $F(\mathbf s)$ is absolutely convergent for $\mathbf s$ such that $\Re s_j > \alpha_j$.
\item[(P2)] There are $n$ linear forms $\mathcal L = \{ \ell_i \}_{i=1}^n$ on $\bC^r$ with the property that $\ell_i(e_j) \in [0, +\infty)$ such that the function $H$ on $\bC^r$ defined by $H(\mathbf s) = F(\mathbf s + \mathbf \alpha) \prod_{i=1}^n \ell_i(\mathbf s)$ has an analytic continuation to a holomorphic function on the domain $\mathcal D(\delta_1) = \{ \mathbf s \in \bC^r \mid \Re \ell_i(\mathbf s) > - \delta_1 \}$ with $\delta_1 >0$. 
\item[(P3)] There is $\delta_2>0$ such that for each $\epsilon, \epsilon'>0$, we have 
$$
|H(\mathbf s)| \ll \prod_{i=1}^n ( |\Im \ell_i(\mathbf s)|+1)^{1-\delta_2 \min (0, \Re \ell_i(\mathbf s))} (1 + 
\| \Im \mathbf s\|_1^\epsilon)
$$
uniformly on the domain $\mathcal D(\delta_1-\epsilon')$. Here for an $r$-tuple of real numbers $\mathbf \tau = (\tau_1, \dots,  \tau_r)$, we have set $\|\mathbf \tau \|_1  = \sum_i |\tau_i|$.
\end{enumerate} 
Suppose $\mathbf \beta = (\beta_1, \dots, \beta_r) \in (0, +\infty)^r$.  Then there is a polynomial $Q = Q_\mathbf \beta$ with real coefficients of degree less than or equal to $n - \rank \mathcal L$ and a positive real number $\vartheta$, depending on $\mathcal L$, $\delta_1$, $\delta_2$, $\mathbf \alpha$, and $\mathbf \beta$ such that for $B \geq 1$ we have 
$$
S_\mathbf \beta(B) : = \sum_{1 \leq d_1 \leq B^{\beta_1}} \cdots \sum_{1 \leq d_r \leq B^{\beta_r}} f(d_1, \dots, d_r)
$$
$$
= B^\sigma Q(\log B) + O(B^{\sigma - \vartheta}), 
$$
with $\sigma = \sum_{j=1}^m \alpha_j \beta_j$. 
\end{theorem}

De la Bret\`{e}che \cite{dela1, dela2} states this theorem only for $\mathcal S_i = \bN$, $1 \leq i \leq r$, but the same argument works for the more general theorem stated here. 

\subsection{Application}
For $e \in \Sigma_1$, we let $D_e$ be the corresponding $T$-invariant divisor. Define a $\varphi_e \in PL(\Sigma)$ by setting $\varphi_e(f) = \delta_{ef}$, 
Kronecker's delta, for $f \in \Sigma_1$.  The line bundle corresponding to $D_e$ is $L_{\varphi_e}$. We denote the metrized line bunde $L_{\varphi_e}$ by $\mathcal L(e)$.  We choose the metrization so that for $x \in X(F)$, 
$$
H_{\mathcal L(e)}(x) = H_\Sigma(x, \varphi_e)^{-1}. 
$$

\begin{theorem}\label{thm:main}
Fix positive real numbers $\beta_e$, for each $e \in \Sigma_1$. For a positive real number $B > 1$, set 
$$
N_T(B) = \# \{ x \in T(F) \mid 1 \leq H_{\mathcal L(e)}(x) \leq B^{\beta_e}, \forall e \in \Sigma_1 \}.
$$
There is a constant $C >0$ and $\epsilon > 0$ such that 
$$
N_T(B) = C B^{\sum_{e \in \Sigma_1} \beta_e} (1 + O(B^{-\epsilon})
$$
as $B \to \infty$. 
\end{theorem}
\begin{proof}
By Theorem \ref{thm:dela} we need to understand the analytic properties of the zeta function 
$$
Z((s_e)_{e \in \Sigma_1}) = \sum_{x \in T(F)} \frac{1}{\prod_{e \in \Sigma_1} H_{\mathcal L(e)}(x)^{s_e}}.
$$
We have 
$$
\prod_{e \in \Sigma_1} H_{\mathcal L(e)}(x)^{s_e} = \prod_{e \in \Sigma_1} H_\Sigma(x, \varphi_e)^{-s_e} 
= H_\Sigma(x, \sum_{e \in \Sigma_1} s_e \varphi_e)^{-1}. 
$$
As a result, 
$$
Z((s_e)_{e \in \Sigma_1}) = Z_\Sigma(\varphi)
$$
with $\varphi = \sum_{e \in \Sigma_1} s_e \varphi_e$. That the conditions required by Theorem \ref{thm:dela} are satisfied is the content of Proposition \ref{prop:dela}.  The only fact that requires verification is the statement that, in the notations of Theorem \ref{thm:dela}, $\deg Q=0$. By Theorem \ref{thm:dela}, $\deg Q =  n - \rank \mathcal L$, but by Proposition \ref{prop:dela} our linear forms are linearly independent and as a result $n = \rank \mathcal L$. 
\end{proof}

\begin{corollary}\label{coro:main}
Fix positive real numbers $\beta_e$, $e \in \Sigma_1$. For a positive real number $B > 1$, set 
$$
N_X(B) = \# \{ x \in X(F) \mid 1 \leq H_{\mathcal L(e)}(x) \leq B^{\beta_e}, e \in \Sigma_1 \}.
$$
There is a constant $C >0$ and $\epsilon' > 0$ such that 
$$
N_X(B) = C B^{\sum_{e \in \Sigma_1} \beta_e} (1 + O(B^{-\epsilon'})
$$
as $B \to \infty$. 
\end{corollary}

\begin{proof}
We will do induction on $d=\dim X$. The key observation is that $X = T \cup \cup_{e \in \Sigma_1} D_e$. If $d=1$, then $X$ is the union of $T$ with a finite set, so the result follows from Theorem \ref{thm:main}. Now suppose the result is true for all toric varieties of dimension $d-1$. Then 
$$
\# \{ x \in D_e(F) \mid 1 \leq H_{\mathcal L(f)}(x) \leq B^f, f \in \Sigma_1 \} 
$$
$$
\leq \# \{ x \in D_e(F) \mid 1 \leq H_{\mathcal L(f)}(x) \leq B^f, f \in \Sigma_1 \cap \Sigma_e\} 
$$
$$
\ll B^{\sum_{f \in \Sigma_1 \cap \Sigma_e} \beta_f}.
$$
As $e \not\in \Sigma_1 \cap \Sigma_f$, the result follows with any $\epsilon'$ less than the minimum of the $\epsilon$ from Theorem \ref{thm:main} and the numbers $\beta_e$, $e \in \Sigma_1$. 
\end{proof}

{\em Address:} Department of Mathematics, Statistics, and Computer Science, University of Illinois at Chicago, 851 S Morgan St, Chicago, IL 60607. 

\

{\em Email:} ademir3@uic.edu 

{\em Email:} rtakloo@uic.edu 

\end{document}